\documentclass{amsart}
\newtheorem{lemma}{Lemma}

\newcommand{\en}{\mathbb{N}}
\newcommand{\qu}{\mathbb{Q}}
\newcommand{\zet}{\mathbb{Z}}
\newcommand{\er}{\mathbb{R}}
\newcommand{\ce}{\mathbb{C}}
\newcommand{\be}{\mathcal{B}}
\newcommand{\alf}{\mbox{\boldmath{$\alpha$}}}
\newtheorem{theorem}{Theorem}
\newtheorem{corollary}{Corollary}
\begin{document}
\title{Weyl's inequality and systems of forms}
\author{Rainer Dietmann}
\address{Department of Mathematics, Royal Holloway, University of London\\
TW20 0EX Egham, United Kingdom}
\thanks{This work has been supported by grant EP/I018824/1 `Forms in many
variables'.}
\email{Rainer.Dietmann@rhul.ac.uk}
\begin{abstract}
Let $F_1, \ldots, F_r$ be integer forms of degree $d \ge 2$ in
$s$ variables.
Relaxing the non-singularity condition in a well known result by Birch,
we establish the expected Hardy-Littlewood asymptotic formula
for the density of integer points on the intersection $F_1 = \ldots =
F_r = 0$, providing that
\[
  s>\max_{\mathbf{a} \in \zet^r \backslash \{\mathbf{0}\}}
  \dim V_\mathbf{a}^* + r(r+1) (d-1) 2^{d-1},
\]
where 
$V_\mathbf{a}^* = \{\mathbf{x} \in \ce^s:
\nabla(a_1 F_1(\mathbf{x}) + \ldots + a_r F_r(\mathbf{x})) =
\mathbf{0}\}$.
In the same context, we also
improve on previous work by Schmidt and show that the expected
Hardy-Littlewood asymptotic formula holds true, providing that
each form in the rational pencil of $F_1, \ldots, F_r$ has
\begin{itemize}
  \item rank exceeding $2r^2+2r \; (d=2)$,
  \item $h$-invariant exceeding $8r^2+8r \; (d=3)$,
  \item $h$-invariant exceeding $\varphi(d) (d-1) 2^{d-1} r(r+1)$ for
  a certain function $\varphi(d)$ when $d \ge 4$.
\end{itemize}
In particular, if $F_1, \ldots, F_r$ are rational quadratic forms, each form
in their complex pencil has rank exceeding $2r^2+2r$, and the intersection
$F_1=\ldots=F_r=0$ has a non-singular real zero, then this intersection also
has a non-trivial rational zero. For $r=1$, this recovers a classical
result by Meyer.\\
Our new tool, which is of interest in itself, is a variant of Weyl's inequality
for general systems of forms which is useful in situations like those above
where one knows a certain lower rank (or dimension of singular locus)
bound for all forms in the rational pencil of the given ones.
\end{abstract}
\maketitle
\section{Introduction}
\label{mallorca}
Let $F_1, \ldots, F_r \in \zet[X_1, \ldots, X_s]$ be forms of degree $d$,
and let $V^*$ be the union of the loci of singularities of the varieties
\[
  F_1(\mathbf{x}) = \mu_1, \ldots, F_r(\mathbf{x}) = \mu_r;
\]
in the following all dimensions are meant to be affine.
Then Birch (\cite{B}) has shown that if
\begin{equation}
\label{virginia_water}
  s > \dim V^* + r(r+1)(d-1)2^{d-1},
\end{equation}
then the expected Hardy-Littlewood asymptotic formula
\[
  \mathfrak{J} \mathfrak{S} P^{s-rd} + O(P^{s-rd-\delta})
\]
for the number of integer solutions
$\mathbf{x} \in \zet^s$ of the system
\[
  F_1(\mathbf{x}) = \ldots = F_r(\mathbf{x}) = 0
\]
holds true, where the $\mathbf{x}$ are constrained to an expanding box
of size $P$, and $\mathfrak{J}$ and $\mathfrak{S}$
are the singular integral and the singular series, respectively.
Alternatively, $V^*$ can also be described as the variety of all
points $\mathbf{x} \in \ce^s$ for which the rank of the Jacobi matrix
\[
  \left( \frac{\partial F_i(\mathbf{x})}{\partial x_j}
  \right)_{\substack{1 \le i \le r\\1 \le j \le s}}
\]
is less than $r$ (see also \cite{AM}).
For $r=1$, $\dim V^*$ is just the dimension
of the singular locus of the variety
\[
  V=\{\mathbf{x} \in \ce^s:
  F_i(\mathbf{x})=0 \;\;\; (1 \le i \le r)\},
\]
but for $r>1$ the quantity $\dim V^*$ can exceed
the dimension of the singular locus of $V$: For example, suppose that
$V$ is non-singular. Then it is still possible to find
$a_1, \ldots, a_r \in \ce$,
not all zero, for which the discriminant
$\operatorname{disc}(a_1 F_1 + \ldots + a_r F_r)$ is zero. This
implies that there is an $\mathbf{x} \in \ce^n \backslash
\{\mathbf{0}\}$ for which
\[
  \nabla(a_1 F_1(\mathbf{x}) + \ldots + a_r F_r(\mathbf{x})) =
  a_1 \nabla F_1(\mathbf{x}) + \ldots + a_r \nabla F_r(\mathbf{x}) =
  \mathbf{0},
\]
whence $\mathbf{x} \in V^*$ and $\dim V^* \ge 1$.
As another example, consider a system of $r$ diagonal forms
$F_1, \ldots, F_r$ of degree $d$ with non-singular coefficient matrix.
Then $V$ is non-singular, but by applying row operations we can
find a form in the pencil of $F_1, \ldots, F_r$ with $r-1$ coefficients
being zero, which shows that $\dim V^* \ge r-1$.\\
Our first aim
in this paper is to show that Birch's condition
\eqref{virginia_water} can be relaxed to the effect that
$\dim V^*$ can be replaced by
\[
  \max_{\mathbf{a} \in \zet^r \backslash \{\mathbf{0}\}}
  \dim V_\mathbf{a}^*,
\]
where
\[
  V_\mathbf{a}^* = \{\mathbf{x} \in \ce^s:
  \nabla(a_1 F_1(\mathbf{x}) + \ldots + a_r F_r(\mathbf{x})) =
  \mathbf{0}\}.
\]
In particular, for $\mathbf{a} \ne \mathbf{0}$ we have
$V_\mathbf{a}^* \subset V^*$, so
\[
  \max_{\mathbf{a} \in \zet^r \backslash \{\mathbf{0}\}}
  \dim V_\mathbf{a}^* \le \dim V^*,
\]
and for $r=1$ the two quantities are the same.
However, for $r>1$ strict inequality
\[
  \max_{\mathbf{a} \in \zet^r \backslash \{\mathbf{0}\}}
  \dim V_\mathbf{a}^* < \dim V^*
\]
is possible, as the left hand side only involves the rational
pencil of $F_1, \ldots, F_r$, whereas the right hand side 
implicitly also includes
the complex one. For instance, suppose that $F_1, \ldots, F_r$
have the property that the equation
\begin{equation}
\label{ente}
  \operatorname{disc}(a_1 F_1 + \ldots + a_r F_r) = 0
\end{equation}
has no solution $\mathbf{a} \in \zet^r \backslash \{\mathbf{0}\}$;
one example for $r=d=2, s=13$ which can be easily checked by a computer
algebra package would be given by the pair of quadratic forms
\begin{align*}
  F_1(X_1, \ldots, X_{13}) & = X_1^2 + \ldots + X_9^2
  + X_{10}^2 + 2 X_{11}^2 + X_{12}^2 + 2 X_{13}^2,\\
  F_2(X_1, \ldots, X_{13}) & =
  2 \sum_{j=0}^2 \left( X_{3j+1} X_{3j+2} + 3 X_{3j+1} X_{3j+3} +
  2 X_{3j+2} X_{3j+3} \right)\\
  & + 2 X_{10} X_{11} + 2 X_{12} X_{13}.
\end{align*}
Then \eqref{ente} having no solution $\mathbf{a} \in \zet^r
\backslash \{\mathbf{0}\}$ implies that
\[
  \max_{\mathbf{a} \in \zet^r \backslash \{\mathbf{0}\}}
  \dim V_\mathbf{a}^* = 0,
\]
whereas as shown above $\dim V^* \ge 1$ for $r>1$,
even if the variety $V$ is non-singular.
Summarising these observations, we therefore find that the following result
is always of the same quality as Birch's one given by
\eqref{virginia_water} regarding the number of required variables,
and for $r>1$ can even be stronger as illustrated by the example
above.
\begin{theorem}
\label{stammtisch}
Let $F_1, \ldots, F_r \in \zet[X_1, \ldots, X_s]$ be forms of degree
$d \ge 2$, and suppose that
\[
  s>\max_{\mathbf{a} \in \zet^r \backslash \{\mathbf{0}\}}
  \dim V_\mathbf{a}^* + r(r+1) (d-1) 2^{d-1}.
\]
Then if $\be$ is a box in $\er^s$ with sides parallel to the coordinate
axes, and contained in the unit box, then there exists a positive $\delta$
such that for the quantity
\[
  \varrho(P) := \#\{\mathbf{x} \in \zet^s:\mathbf{x} \in P\be \mbox{ and }
  F_i(\mathbf{x}) = 0 \; (1 \le i \le r)\}
\]
the asymptotic formula
\begin{equation}
\label{unknown}
  \varrho(P) = \mathfrak{J} \mathfrak{S} P^{s-rd} + O(P^{s-rd-\delta})
\end{equation}
holds true, where $\mathfrak{J}$ and $\mathfrak{S}$
are the singular integral and the singular series, respectively.
\end{theorem}
For instance, as shown by the example above, Theorem \ref{stammtisch}
handles `generic' pairs of quadratic forms in $13$ variables; see
also \cite{Munshi} for recent work by Munshi which gives a stronger
result in this special case $r=d=2$.\\
Usually, $V^*$ as well as the $V_\mathbf{a}^*$ are difficult to understand,
and one would prefer a condition
which is arithmetically easier to handle.
This point of view was taken up by Schmidt
(\cite{S1}, \cite{S2}) for the case of quadratic and cubic forms,
and in \cite{Schmidt} for forms of any degree.
For a system of $r$ rational quadratic forms, he could show that an
asymptotic formula for the number of integer zeros in an expanding region
holds true, provided that every form in the rational pencil has rank
exceeding $2r^2+3r$. Birch's condition (\ref{virginia_water}) for $d=2$
reads $s>\dim V^*+2r^2+2r$, so one might wonder if Schmidt's rank bound
$2r^2+3r$ can be relaxed to $2r^2+2r$. This is indeed the case, as
illustrated by our first result.
\begin{theorem}
\label{see}
Let $F_1, \ldots, F_r \in \zet[X_1, \ldots, X_s]$ be quadratic forms,
such that each form in their rational pencil has rank exceeding $2r^2+2r$.
Then if $\be$ is a box in $\er^s$ with sides parallel to the coordinate
axes, and contained in the unit box, then there exists a positive $\delta$
such that for the quantity
\[
  \varrho(P) := \#\{\mathbf{x} \in \zet^s:\mathbf{x} \in P\be \mbox{ and }
  F_i(\mathbf{x}) = 0 \; (1 \le i \le r)\}
\]
the asymptotic formula
\begin{equation}
\label{unknown}
  \varrho(P) = \mathfrak{J} \mathfrak{S} P^{s-2r} + O(P^{s-2r-\delta})
\end{equation}
holds true, where $\mathfrak{J}$ and $\mathfrak{S}$
are the singular integral and the singular series, respectively.
\end{theorem}
In particular, under the rank condition of Theorem \ref{see},
a system $F_1(\mathbf{x}) = \ldots = F_r(\mathbf{x}) = 0$ of rational
quadratic forms satisfies the Local-Global principle. If one imposes further
conditions, one can show that $\mathfrak{J}$ and $\mathfrak{S}$ are positive and
concludes that there are non-trivial rational zeros. The following result
is along these lines, improving a $2r^2+3r$ which has previously
been known (see Theorem 1 in \cite{D}) to $2r^2+2r$.
\begin{corollary}
\label{akw}
Let $F_1, \ldots, F_r$ be rational quadratic forms. Suppose that each form in
the complex pencil of $F_1, \ldots, F_r$ has rank exceeding $2r^2+2r$, and
suppose that the system $F_1=\ldots=F_r=0$ has a non-singular real zero.
Then the asymptotic formula (\ref{unknown}) holds true, and the singular
integral $\mathfrak{J}$ and the singular series $\mathfrak{S}$ are positive.
In particular, the intersection $F_1=\ldots=F_r=0$ has a non-trivial rational
zero.
\end{corollary}
\begin{proof}
Since each form in the complex pencil of $F_1, \ldots, F_r$ has rank
exceeding $2r^2+2r$, also each form in the rational pencil of $F_1, \ldots,
F_r$ has rank exceeding $2r^2+2r$. Hence, by Theorem \ref{see}, the
asymptotic formula (\ref{unknown}) holds true. Moreover, for the same
reason (see \cite{D} or \cite{S1} for details), each form in any
$\qu_p$-rational pencil of $F_1, \ldots, F_r$ has rank exceeding
$2r^2+2r$ as well.
As shown in \cite{D}, p. 510, this implies that $\mathfrak{S}>0$.
Finally, the assumption on non-singular real solubility of $F_1=\ldots
=F_r=0$ implies that $\mathfrak{J}>0$.
\end{proof}
Since a singular rational quadratic form trivially has a non-trivial rational
zero, Corollary \ref{akw} for $r=1$ recovers a classical result by Meyer
\cite{M}: Any indefinite rational quadratic form in at least five variables
has a non-trivial rational zero. Much more than Corollary \ref{akw}
is known for $r=2$ (see \cite{C}, \cite{HB}),
but for $r>2$ our result seems to be the strongest available at
present.\\
To state our result for systems of cubic forms, we first have to introduce
the following terminology: If $C(X_1, \ldots, X_s) \in \qu[X_1, \ldots, X_s]$
is a rational cubic form, then its $h$-\textit{invariant} $h(C)$ is the
smallest non-negative integer $h$ such that $C$ can be written in the form
\[
  C(X_1, \ldots, X_s) = \sum_{i=1}^h L_i(X_1, \ldots, X_s)
  Q_i(X_1, \ldots, X_s)
\]
for suitable linear forms $L_i \in \qu[X_1, \ldots, X_s]$
and quadratic forms $Q_i \in \qu[X_1, \ldots, X_s]$.
One can think of the $h$-invariant as some way of generalising rank from
quadratic to cubic forms. For a system of $r$ rational cubic forms,
Schmidt (\cite{S2}, Theorem 2)
has shown that an asymptotic formula for the number of
integer zeros in an expanding region holds true, provided that each form in
the rational pencil of $C_1, \ldots, C_r$ has $h$-invariant exceeding
$10r^2+6r$. Again, Birch's condition (\ref{virginia_water}) for $d=3$
suggests that a weaker bound, namely $8r^2+8r$ could suffice, which indeed
is true.
\begin{theorem}
\label{rumpsteak}
Let $F_1, \ldots, F_r \in \zet[X_1, \ldots, X_s]$ be cubic forms,
such that each form in their rational pencil has $h$-invariant
exceeding $8r^2+8r$.
Then if $\be$ is a box in $\er^s$ with sides parallel to the coordinate
axes, and contained in the unit box, then there exists a positive $\delta$
such that for the quantity
\[
  \varrho(P) := \#\{\mathbf{x} \in \zet^s:\mathbf{x} \in P\be \mbox{ and }
  F_i(\mathbf{x}) = 0 \; (1 \le i \le r)\}
\]
the asymptotic formula
\[
  \varrho(P) = \mathfrak{J} \mathfrak{S} P^{s-3r} + O(P^{s-3r-\delta})
\]
holds true, where $\mathfrak{J}$ and $\mathfrak{S}$
are the singular integral and the singular series, respectively.
\end{theorem}
Not surprisingly, the proofs of Theorem \ref{stammtisch},
Theorem \ref{see} and Theorem \ref{rumpsteak}
make use of the Hardy-Littlewood circle method.
Our main innovation is a variant of Birch's form of Weyl's inequality
for general systems of forms which turns out to be more useful in the specific
situations encountered here, where one knows a certain lower rank
(or dimension of singular locus) bound for
all forms in the rational pencil of the given ones.\\
One might ask if our method also provides an analogue
of Theorem \ref{see} and Theorem \ref{rumpsteak}
for systems of rational forms of
degree $d \ge 4$. This is indeed the case, and the treatment is analogous
to that for $d=3$, so we only give a brief sketch:
The concept of an $h$-invariant generalises from cubic forms to higher
degree forms, see Schmidt's seminal
work \cite{Schmidt}, p. 245. Using one of the key results in \cite{Schmidt},
one can show that if
$F_1, \ldots, F_r \in \zet[X_1, \ldots, X_s]$ are forms of degree $d$,
and each form in their rational pencil has $h$-invariant exceeding
\begin{equation}
\label{oktober}
  \varphi(d) (d-1) 2^{d-1} r(r+1),
\end{equation}
where $\varphi(2)=\varphi(3)=1$, $\varphi(4)=3$, $\varphi(5)=13$, and
$\varphi(d)<(\log 2)^{-d} d!$ in general, then an asymptotic formula for
the number of integer zeros of the system $F_1=\ldots=F_r=0$ in an
expanding box of size $P$ of the form
\[
  \mathfrak{J} \mathfrak{S} P^{s-rd} + O(P^{s-rd-\delta})
\]
holds true, parallelising Theorems \ref{see} and \ref{rumpsteak}
(note that also Theorem II in \cite{Schmidt} deals with generic systems of
forms of not necessarily the same degrees and could be specialised
to the setting here, but is not optimised to our situation of all forms
having the same degree $d$).
However, for $d \ge 4$ the bound \eqref{oktober} no longer parallels the
corresponding term $r(r+1)(d-1)2^{d-1}$ in Birch's result
\eqref{virginia_water} as $\varphi(d)>1$ for $d \ge 4$.\\ \\
\textbf{Acknowledgment:} The author would like to thank the referees
for carefully reading this paper.\\
Since this paper was submitted for publication,
Schindler \cite{Schindler}
has posted a preprint which independently proves Theorem \ref{stammtisch}
and the result described around equation \eqref{oktober} above
which is proved in section \ref{thursday}.
\section{Weyl's inequality}
\label{chile}
Let $F_1, \ldots, F_r \in \zet[X_1, \ldots, X_s]$ be forms of degree
$d \ge 2$.
We can write $F_i$ in the form
\[
  F_i(\mathbf{x}) = \sum_{1 \le j_1, \ldots, j_d \le s}
  c_{j_1, \ldots, j_d}^{(i)} x_{j_1} \cdots x_{j_d} \;\;\; (1 \le i \le r),
\]
and for the purpose of studying the system $F_1 = \ldots = F_r = 0$
we may without loss of generality assume that the coefficients
$c_{j_1, \ldots, j_d}^{(i)}$ are symmetric in $j_1, \ldots, j_d$.
Let $\be$ be an $s$-dimensional box with sides parallel to the coordinate
axes, and
for $\alf \in \er^r$ let $S(\mathbf{\alf})$ be the exponential sum
\[
  S(\alf) = \sum_{\mathbf{x} \in P\be} e(\alpha_1 F_1(\mathbf{x}) +
  \ldots + \alpha_r F_r(\mathbf{x})).
\]
Moreover, for $\alf \in \er^r$, $\mathbf{x}^{(1)}, \ldots,
\mathbf{x}^{(d-1)} \in \zet^s$ and $j \in \{1, \ldots, s\}$ let
\begin{equation}
\label{bookie}
  \Phi_j(\mathbf{\alf}; \mathbf{x}^{(1)}, \ldots, \mathbf{x}^{(d-1)}) =
  \sum_{i=1}^r \alpha_i \Psi_j^{(i)}(\mathbf{x}^{(1)}, \ldots,
  \mathbf{x}^{(d-1)})
\end{equation}
where
\[
  \Psi_j^{(i)}(\mathbf{x}^{(1)}, \ldots, \mathbf{x}^{(d-1)}) =
  d! \sum_{1 \le j_1, \ldots, j_{d-1} \le s}
  c^{(i)}_{j, j_1, \ldots, j_{d-1}} x_{j_1}^{(1)} \cdots x_{j_{d-1}}^{(d-1)}.
\]
Finally, let
\begin{eqnarray}
\label{gerrard}
  N(P;Q;\alf) & = & \#\{\mathbf{x}^{(1)}, \ldots, \mathbf{x}^{(d-1)}
  \in \zet^s: \mathbf{x}^{(i)} \in P\be \; (1 \le i \le d-1) \\
  & & \mbox{and } ||\Phi_j(\alf; \mathbf{x}^{(1)}, \ldots,
  \mathbf{x}^{(d-1)})|| < Q \; (1 \le j \le s)\}, \nonumber
\end{eqnarray}
where $||\cdot||$ as usual denotes distance to the nearest integer.
In the following, all implied $O$-constants depend at most on
$s, r, d, F_1, \ldots, F_r, \be$ and a chosen small positive $\varepsilon$.
\begin{lemma}
\label{england}
Let $0 < \theta \le 1$, $\varepsilon>0$ and $k>0$. Then if
\[
  S(\mathbf{\alf}) > P^{s-k},
\]
then
\begin{equation}
\label{rooney}
  N(P^{\theta};P^{-d+(d-1)\theta};\alf) \gg P^{(d-1)s\theta-2^{d-1}k-
  \varepsilon}.
\end{equation}
\end{lemma}
\begin{proof}
This is Lemma 2.4 in \cite{B}.
\end{proof}
Our next lemma can be thought of as a variant of Lemma 2.5 in \cite{B},
where alternative (iii) has been replaced by one more suitable for dealing
with systems of forms satisfying a certain pencil condition.
\begin{lemma}
\label{booker}
In the notation of Lemma \ref{england}, we either (i) have
\[
  S(\mathbf{\alf}) \ll P^{s-k},
\]
or (ii) there are integers $a_1, \ldots, a_r, q$ such that
\begin{eqnarray*}
  (a_1, \ldots, a_r, q)=1, & & \\
  |q\alpha_i-a_i| \ll P^{-d+r(d-1)\theta} \;\;\; (1 \le i \le r), & & \\
  1 \le q \ll P^{r(d-1)\theta}, & & \\
\end{eqnarray*}
or (iii) there are integers $a_1, \ldots, a_r$, not all zero, such that
\[
  M(a_1, \ldots, a_r; P^\theta)
  \gg (P^\theta)^{(d-1)s-2^{d-1}k/\theta-\varepsilon}
\]
where
\begin{eqnarray}
\label{odds}
  M(a_1, \ldots, a_r; H) & = & \#\{\mathbf{x}^{(1)}, \ldots,
  \mathbf{x}^{(d-1)} \in \zet^s: \mathbf{x}^{(i)} \in H\be \; (1 \le i
  \le d-1) \\
  & & \mbox{and } \Phi_j(\mathbf{a}; \mathbf{x}^{(1)}, \ldots,
  \mathbf{x}^{(d-1)}) = 0 \; (1 \le j \le s)\}. \nonumber
\end{eqnarray}
\end{lemma}
\begin{proof}
Our proof is similar to that of Lemma 2.5 in \cite{B} (see also the
remark at the bottom of page 248 in \cite{B}).
Suppose that alternative (i) is false. Then by Lemma \ref{england},
the lower bound (\ref{rooney}) holds true. Let $\Psi$ be the matrix whose
elements are the numbers $\Psi_j^{(i)}(\mathbf{x}^{(1)}, \ldots,
\mathbf{x}^{(d-1)})$, where the column $i$ ranges from $1$ to $r$, and
the rows range over the Cartesian product of all possible choices for
$j \in \{1, \ldots, s\}$ times all possible choices for tuples
$\mathbf{x}^{(1)}, \ldots, \mathbf{x}^{(d-1)}$ counted by $N$ in
(\ref{rooney}).\\ \\
\textbf{Case I:} rank $\Psi=r$.\\ \\
Then $\Psi$ has a non-singular $r \times r$ submatrix $R$.
Each row of $R$ is of the form
\[
  \left( \Psi_j^{(1)}(\mathbf{x}^{(1)}, \ldots, \mathbf{x}^{(d-1)}),
  \ldots, \Psi_j^{(r)}(\mathbf{x}^{(1)}, \ldots, \mathbf{x}^{(d-1)}) \right)
\]
for suitable $j \in \{1, \ldots, s\}$ and $\mathbf{x}^{(1)}, \ldots,
\mathbf{x}^{(d-1)} \in P^\theta \be$ counted by $N$ in (\ref{rooney}).
In particular, all matrix elements $r_{\ell i}$ are integers and
\begin{equation}
\label{klose}
  |r_{\ell i}| \ll P^{\theta (d-1)}
\end{equation}
for all $i \in \{1, \ldots, r\}$ and for all $\ell$.
By (\ref{rooney}) and the definition (\ref{gerrard}) of $N$, we have
\begin{equation}
\label{podolski}
  R \left( \begin{array}{c} \alpha_1 \\ \vdots \\ \alpha_r \end{array} \right)
  = \left( \begin{array}{c} b_1+c_1 \\ \vdots \\ b_r+c_r \end{array}
  \right),
\end{equation}
where the $b_i$ are integers and $c_i \in \er$ such that
\begin{equation}
\label{uruguay}
  |c_i| \ll P^{-d+(d-1)\theta} \;\;\; (1 \le i \le r).
\end{equation}
Now let $q = \det R$. Then $q \ne 0$ and
\[
  |q| \ll P^{r(d-1)\theta}
\]
by (\ref{klose}). Let $\mathbf{a} \in \er^r$ be the solution of
\begin{equation}
\label{mueller}
  R \left( \begin{array}{c} a_1 \\ \vdots \\ a_r \end{array} \right)
  = q \left( \begin{array}{c} b_1 \\ \vdots \\ b_r \end{array} \right).
\end{equation}
By Cramer's rule, $\mathbf{a} \in \zet^r$. Now (\ref{podolski}) and
(\ref{mueller}) yield
\begin{equation}
\label{neuner}
  q \left( \begin{array}{c} \alpha_1 \\ \vdots \\ \alpha_r \end{array}
  \right) - \left( \begin{array}{c} a_1 \\ \vdots \\ a_r \end{array}
  \right) =
  q R^{-1} \left( \begin{array}{c} c_1 \\ \vdots \\ c_r \end{array} \right).
\end{equation}
Moreover, by (\ref{klose}) and Cramer's rule all elements in the matrix
$qR^{-1}$ are at most $O(P^{\theta(r-1)(d-1)})$. Hence (\ref{uruguay})
and (\ref{neuner}) give
\[
  |q\alpha_i-a_i| \ll P^{-d+r(d-1)\theta} \;\;\; (1 \le i \le r).
\]
Clearly, by multiplying with $-1$ if necessary we can ensure that $q \in \en$,
and by dividing through $(a_1, \ldots, a_r, q)$ we can achieve
$(a_1, \ldots, a_r, q)=1$ without affecting the quality of the
Diophantine approximations to $\alpha_1, \ldots, \alpha_r$. Hence (ii) is
true.\\ \\
\textbf{Case II:} rank $\Psi<r$.\\ \\
In this case, the columns of $\Psi$ must be linearly dependent.
Therefore, there exists $\mathbf{a} \in \zet^r \backslash\{\mathbf{0}\}$
such that
\[
  \sum_{i=1}^r a_i
  \Psi_j^{(i)}(\mathbf{x}^{(1)}, \ldots, \mathbf{x}^{(d-1)})=0
\]
for all  $j \in \{1, \ldots, s\}$ and for all $\mathbf{x}^{(1)}, \ldots,
\mathbf{x}^{(d-1)} \in \zet^s$ that are counted by $N$ in (\ref{rooney}).
By (\ref{bookie}), (\ref{rooney}) and (\ref{odds}) we immediately get
the conclusion in alternative (iii).
\end{proof}
\section{Birch's Theorem revisited}
The following result is along the lines of \S3 in \cite{B} and shows that
alternative (iii) in Lemma \ref{booker} can be excluded if the variety under
consideration is sufficiently non-singular.
\begin{lemma}
\label{neujahr}
Keeping the notation of sections \ref{mallorca} and \ref{chile},
let $\mathbf{a} \in \zet^r \backslash\{\mathbf{0}\}$. Then
\[
  M(a_1, \ldots, a_r; H) \ll H^{(d-1)s-
  \operatorname{codim}V_{\mathbf{a}}^*},
\]
where the implied $O$-constant does not depend on $\mathbf{a}$.
\end{lemma}
\begin{proof}
Following the proof of Lemma 3.2 in \cite{B}, we have to consider the
variety
\[
  W = \{\mathbf{x}=(\mathbf{x}^{(1)}, \ldots, \mathbf{x}^{(d-1)})
  \in \ce^{(d-1)s}:
  \Phi_j(\mathbf{a}; \mathbf{x}^{(1)}, \ldots, \mathbf{x}^{(d-1)})=0
  \;\;\; (1 \le j \le s)\}. 
\]
Then
\[
  M(a_1, \ldots, a_r; H) \ll H^{\dim W},
\]
where the implied $O$-constant depends at most on $s$, $r$ and $d$.
Now let $U$ be the diagonal
\[
  U = \{\mathbf{x}=(\mathbf{x}^{(1)}, \ldots, \mathbf{x}^{(d-1)})
  \in \ce^{(d-1)s}:
  \mathbf{x}^{(1)} = \ldots = \mathbf{x}^{(d-1)}\}.
\]
Then
\[
  \dim (W \cap U) \ge \dim W + \dim U - (d-1)s
  = \dim W - (d-2)s,
\]
and for each $(\mathbf{x}, \ldots, \mathbf{x}) \in W \cap U$
we have $\nabla (a_1 F_1(\mathbf{x}) + \ldots + a_r F_r(\mathbf{x})) =
\mathbf{0}$, so $\dim (W \cap U) \le \dim V_\mathbf{a}^*$ and
\[
  \dim W \le (d-2)s + \dim V_\mathbf{a}^*
  = (d-1)s - \operatorname{codim} V_\mathbf{a}^*,
\]
which finishes the proof.
\end{proof}
Our next result is an analogue of the central Lemma 4.3 in \cite{B}.
\begin{lemma}
\label{liederhalle}
Keeping the notation of sections \ref{mallorca} and \ref{chile}, define $K$
by
\[
  \max_{\mathbf{a} \in \zet^r \backslash \{\mathbf{0}\}}
  \dim V_\mathbf{a}^* = s-2^{d-1}K.
\]
Then we either (i) have
\[
  S(\alf) \ll P^{s-K\theta+\varepsilon},
\]
or alternative (ii) of Lemma \ref{booker} holds true.
\end{lemma}
\begin{proof}
Suppose that alternative (iii) of Lemma \ref{booker} is true. Then there
exists $\mathbf{a} \in \zet^r \backslash\{\mathbf{0}\}$ such that
\begin{equation}
\label{beethoven}
  M(a_1, \ldots, a_r; P^\theta)
  \gg (P^\theta)^{(d-1)s-2^{d-1}k/\theta-\varepsilon}.
\end{equation}
On the other hand, by Lemma \ref{neujahr}, we know that
\begin{equation}
\label{schubert}
  M(a_1, \ldots, a_r; P^\theta) \ll
  (P^\theta)^{(d-1)s-\operatorname{codim} V_\mathbf{a}^*}
  \le (P^\theta)^{(d-1)s-2^{d-1}K}.
\end{equation}
The conditions (\ref{beethoven}) and (\ref{schubert}) are only compatible
if $K \le \frac{k}{\theta} + \varepsilon \cdot 2^{1-d}$, so for
$k=K\theta-\varepsilon$ alternative (iii) of Lemma \ref{booker} is
impossible, which proves the result.
\end{proof}
Using Lemma \ref{liederhalle} instead of Lemma 4.3 in \cite{B} and then
following the rest of the arguments in \cite{B} establishes Theorem
\ref{stammtisch}.
\section{Systems of quadratic forms}
For quadratic forms, it is easy to give an interpretation of alternative
(iii) in Lemma \ref{booker}. To this end we need the following elementary
observation.
\begin{lemma}
\label{longport}
Let $L_1, \ldots, L_s \in \zet[X_1, \ldots, X_s]$ be linear forms, such
that their span in the space of linear forms has rank at least $m$. Then,
uniformly in $L_1, \ldots, L_s$, we have
\[
  \#\{\mathbf{x} \in \zet^s:\mathbf{x} \in P\be \mbox{ and }
  L_j(\mathbf{x})=0 \; (1 \le j \le s)\} \ll P^{s-m}.
\]
\end{lemma}
\begin{proof}
This is trivial.
\end{proof}
We can now reformulate Lemma \ref{booker} for systems of quadratic forms.
\begin{lemma}
\label{greenwhich}
Let $d=2$. Suppose that each quadratic form in the rational pencil of
$F_1, \ldots, F_r$ has rank at least $m$. Then, using the notation of
Lemma \ref{england}, we either (i) have
\[
  S(\mathbf{\alf}) \ll P^{s-m\theta/2+\varepsilon},
\]
or alternative (ii) (where $d=2$) of Lemma \ref{booker} holds true.
\end{lemma}
\begin{proof}
Suppose that neither alternative (i) nor alternative (ii) of Lemma
\ref{booker} are true. Then alternative (iii) must be true. Hence there
exists $\mathbf{a} \in \zet^r \backslash \{\mathbf{0}\}$ such that
$M(a_1, \ldots, a_r; P^{\theta}) \gg (P^{\theta})^{s-2k/\theta-\varepsilon}$.
This means that
\begin{equation}
\label{jekyll}
  \#\{\mathbf{x} \in \zet^s:\mathbf{x} \in P^{\theta}\be \mbox{ and }
  \Phi_j(\mathbf{a}; \mathbf{x}) = 0 \; (1 \le j \le s)\}
  \gg (P^{\theta})^{s-2k/\theta-\varepsilon}.
\end{equation}
Now consider the quadratic form
\[
  F = \sum_{i=1}^r a_i F_i.
\]
This form is in the rational pencil of $F_1, \ldots, F_r$, hence
rank $F \ge m$. Clearly
\begin{eqnarray*}
  F(\mathbf{x}) & = & \sum_{i=1}^r a_i F_i(\mathbf{x})
  = \frac{1}{2} \sum_{i=1}^r a_i \sum_{j=1}^s x_j \Psi_j^{(i)}(\mathbf{x})
  \\ & = & \frac{1}{2} \sum_{j=1}^s x_j \sum_{i=1}^r a_i
  \Psi_j^{(i)}(\mathbf{x}) = \frac{1}{2} \sum_{j=1}^s
  x_j \Phi_j(\mathbf{a}; \mathbf{x}).
\end{eqnarray*}
Since rank $F \ge m$, the dimension of the linear space of linear forms
spanned by the $\Phi_j(\mathbf{a}; \mathbf{x}) \; (1 \le j \le s)$
is at least $m$. Hence, by Lemma \ref{longport},
\begin{equation}
\label{hyde}
  \#\{\mathbf{x} \in \zet^s:\mathbf{x} \in P^{\theta}\be \mbox{ and }
  \Phi_j(\mathbf{a}; \mathbf{x}) = 0 \; (1 \le j \le s)\}
  \ll (P^{\theta})^{s-m}.
\end{equation}
The conditions (\ref{jekyll}) and (\ref{hyde}) are only compatible if
$s-2k/\theta-\varepsilon \le s-m$, which implies that
$k \ge m\theta/2-\varepsilon \frac{\theta}{2}$.
Therefore, for $k=m\theta/2-\varepsilon$ alternative
(iii) is impossible, which proves Lemma \ref{greenwhich}.
\end{proof}
\begin{corollary}
In the notation of Lemma \ref{england}, suppose that each quadratic form
in the rational pencil of $F_1, \ldots, F_r$ has rank exceeding $2r^2+2r$.
Moreover, let $0<\Delta\le r$. Then for each $\alf \in [0,1]^r$ we either
(i) have
\begin{equation}
\label{longitude}
  S(\alf) \ll P^{s-\Delta \left(r+1+\frac{1}{2r}\right)+\varepsilon},
\end{equation}
(ii) or there exist integers $a_1, \ldots, a_r, q$ such that
$(a_1, \ldots, a_r, q)=1$, $1 \le q \ll P^\Delta$ and
\[
  |q\alpha_i-a_i| \ll P^{-2+\Delta} \;\;\; (1 \le i \le r).
\]
\end{corollary}
\begin{proof}
This follows immediately from Lemma \ref{greenwhich} on letting
 $\Delta=r\theta$
and $m=2r^2+2r+1$, noting that $\Delta\left(r+1+\frac{1}{2r}\right) =
\frac{1}{2r} \cdot r\theta \cdot m = \frac{m\theta}{2}$.
\end{proof}
Note that this is essentially the same as Lemma 6 in \cite{S1} (our bound
(\ref{longitude}) is slightly stronger than (i) in Lemma 6 in \cite{S1},
under the weaker rank condition $2r^2+2r$ instead of $2r^2+3r$). We can
now proceed exactly as in \cite{S1} to deduce Theorem \ref{see}.
\section{Systems of cubic forms}
\label{maritime}
The cubic case is slightly more difficult. We start our discussion with
the following result going back to Davenport and Lewis.
\begin{lemma}
\label{mappus}
Let $C(X_1, \ldots, X_s) \in \zet[X_1, \ldots, X_s]$ be a cubic form,
given in the form
\[
  C(X_1, \ldots, X_s) = \sum_{1 \le j_1, j_2, j_3 \le s}
  c_{j_1, j_2, j_3} x_{j_1} x_{j_2} x_{j_3}
\]
for symmetric integer coefficients $c_{j_1, j_2, j_3}$. Moreover, let $B_j$
be the bilinear forms
\[
  B_j(\mathbf{x}^{(1)}, \mathbf{x}^{(2)}) = \sum_{1 \le j_1, j_2 \le s}
  c_{j,j_1,j_2} x_{j_1}^{(1)} x_{j_2}^{(2)} \;\;\; (1 \le j \le s).
\]
Then
\begin{align}
\label{regen}
  & \#\{\mathbf{x}^{(1)}, \mathbf{x}^{(2)} \in \zet^s: \mathbf{x}^{(1)},
  \mathbf{x}^{(2)} \in P\be \mbox{ and } B_j(\mathbf{x}^{(1)},
  \mathbf{x}^{(2)}) = 0 \; (1 \le j \le s)\} \nonumber \\
  \ll & P^{2s-h(C)},
\end{align}
where $h(C)$ is the $h$-invariant of $C$ as introduced in section
\ref{mallorca}.
The implied $O$-constant does not depend on $C$.
\end{lemma}
\begin{proof}
This is Lemma 3 in \cite{DL}.
\end{proof}
\begin{lemma}
\label{schiff}
Let $d=3$. Suppose that each cubic form in the rational pencil of
$F_1, \ldots, F_r$ has $h$-invariant at least $m$. Then, using the notation of
Lemma \ref{england}, we either (i) have
\[
  S(\mathbf{\alf}) \ll P^{s-m\theta/4+\varepsilon},
\]
or alternative (ii) (where $d=3$) of Lemma \ref{booker} holds true.
\end{lemma}
\begin{proof}
Suppose that neither alternative (i) nor alternative (ii) of Lemma
\ref{booker} are true. Then by alternative (iii) of that lemma, there exists
$\mathbf{a} \in \zet^r \backslash \{\mathbf{0}\}$ such that
$M(a_1, \ldots, a_r; P^\theta) \gg
\left(P^\theta\right)^{2s-4k/\theta-\varepsilon}$. This means that
\begin{eqnarray}
\label{liam}
  & & \#\{\mathbf{x} \in \zet^s: \mathbf{x}^{(1)}, \mathbf{x}^{(2)} \in
  P^\theta B \mbox{ and } \Phi_j(\mathbf{a}; \mathbf{x}^{(1)},
  \mathbf{x}^{(2)})=0 \; (1 \le j \le s)\} \nonumber
  \\ & \gg & (P^\theta)^{2s-4k/\theta-\varepsilon}.
\end{eqnarray}
Now consider the cubic form
\[
  C = \sum_{i=1}^r a_i F_i.
\]
This form is in the rational pencil of $F_1, \ldots, F_r$, hence $h(C) \ge m$.
Using the notation of Lemma \ref{mappus}, we find that
\[
  c_{j_1,j_2,j_3} = \sum_{i=1}^r a_i c^{(i)}_{j_1, j_2, j_3},
\]
hence
\begin{eqnarray*}
  B_j(\mathbf{x}^{(1)}, \mathbf{x}^{(2)}) & = &
  \sum_{1 \le j_1, j_2 \le s} \sum_{i=1}^r a_i c_{j, j_1, j_2}^{(i)}
  x_{j_1}^{(1)} x_{j_2}^{(2)} \;\;\; (1 \le j \le s)\\
  & = & \sum_{i=1}^r a_i \sum_{1 \le j_1, j_2 \le s} c_{j,j_1,j_2}^{(i)}
  x_{j_1}^{(1)} x_{j_2}^{(2)} \;\;\; (1 \le j \le s)\\
  & = & \frac{1}{6} \sum_{i=1}^r a_i \Psi_j^{(i)}(\mathbf{x}^{(1)},
  \mathbf{x}^{(2)}) = \frac{1}{6} \Phi_j(\mathbf{a}; \mathbf{x}^{(1)},
  \mathbf{x}^{(2)}) \;\;\; (1 \le j \le s).
\end{eqnarray*}
Since $h(C) \ge m$, by Lemma \ref{mappus} we have
\begin{equation}
\label{neeson}
  \#\{\mathbf{x}^{(1)}, \mathbf{x}^{(2)} \in P^\theta \be \mbox{ and }
  B_j(\mathbf{x}^{(1)}, \mathbf{x}^{(2)})=0 \;\;\; (1 \le j \le s)\} \ll
  (P^\theta)^{2s-m}.
\end{equation}
The conditions (\ref{liam}) and (\ref{neeson}) are only compatible if
\begin{equation}
\label{dienstag}
  2s-4k/\theta - \varepsilon \le 2s-m,
\end{equation}
which implies that $k \ge m \theta / 4 - \varepsilon \frac{\theta}{4}$.
As for Lemma
\ref{greenwhich} above, this concludes the proof.
\end{proof}
\begin{corollary}
\label{windsor_farm_shop}
In the notation of Lemma \ref{england}, suppose that each cubic form in the
rational pencil of $F_1, \ldots, F_r$ has $h$-invariant exceeding $8r^2+8r$.
Moreover, let $0<\Delta \le 2r$. Then for each $\alf \in [0,1]^r$ we
either (i) have
\begin{equation}
\label{christopher_grill}
  S(\alf) \ll P^{s-\Delta\left(r+1+\frac{1}{8r}\right)+\varepsilon},
\end{equation}
or (ii) there exist integers $a_1, \ldots, a_r, q$ such that
$(a_1, \ldots, a_r, q)=1$, $1 \le q \ll P^{\Delta}$ and
\[
  |q\alpha_i-a_i|  \ll P^{-3+\Delta} \;\;\; (1 \le i \le r).  
\]
\end{corollary}
\begin{proof}
This follows immediately from Lemma \ref{schiff} on setting $\Delta=2r\theta$
and $m=8r^2+8r+1$, noting that $\Delta\left(r+1+\frac{1}{8r}\right)
=\frac{1}{8r} \cdot 2r\theta \cdot m = \frac{m\theta}{4}$.
\end{proof}
Note that this is essentially the same as Lemma 7 in \cite{S2} (our bound
(\ref{christopher_grill}) is slightly stronger than the bound (i) in
Lemma 7 in \cite{S2}, under the weaker rank condition $8r^2+8r$ instead of
$10r^2+6r$). We can now continue exactly as in \cite{S2} to deduce
Theorem \ref{rumpsteak}.
\section{Systems of higher degree forms}
\label{thursday}
Let us briefly sketch how to obtain the result stated at the end of the
introduction. The argument is almost completely analogous to that in
section \ref{maritime}, all the hard work just being imported from \cite{Schmidt}.
We have to replace bilinear forms by $(d-1)$-linear forms, and to replace
Lemma \ref{mappus} by Proposition III from \cite{Schmidt}, yielding an upper
bound of
\[
  O(P^{(d-1)s-h(F)/\varphi(d)})
\]
instead of \eqref{regen}, where $F$ is our form of degree $d$ under consideration,
and $\varphi(2)=\varphi(3)=1$, $\varphi(4)=3$, $\varphi(5)=13$,
and $\varphi(d)<(\log 2)^{-d} d!$ in general (see Proposition $\text{III}_{\ce}$
in \cite{Schmidt}). One can then modify the proof of Lemma \ref{schiff}
accordingly, obtaining
\[
  (d-1)s-2^{d-1}k/\theta-\varepsilon \le (d-1)s-m/\varphi(d)
\]
instead of \eqref{dienstag}, which yields the analogue of Lemma \ref{schiff}
to the effect that if $F_1, \ldots, F_r$ are integer forms of degree $d$, and
each form in their rational pencil has $h$-invariant at least $m$, then either
\[
  S(\mathbf{\alf}) \ll P^{s-m\theta/(\varphi(d) 2^{d-1})+\varepsilon},
\]
or alternative (ii) of Lemma \ref{booker} holds true. Putting
$\Delta=(d-1)r \theta$ and
\[
  m=\varphi(d)(d-1)2^{d-1} r(r+1) + 1,
\]
one finds that
\[
  \Delta
  \left( r+1+\frac{1}{r \varphi(d) (d-1) 2^{d-1}} \right)
  = \frac{m \theta}{\varphi(d) 2^{d-1}},
\]
giving the corresponding analogue of Corollary \ref{windsor_farm_shop}:
If each form in the rational pencil of $F_1, \ldots, F_r$ has $h$-invariant
exceeding $\varphi(d) (d-1) 2^{d-1} r(r+1)$, then either
\[
  S(\alf) \ll P^{s-\Delta\left(r+1+
  \frac{1}{r \varphi(d) (d-1) 2^{d-1}}\right)+\varepsilon},
\]
or there exist integers $a_1, \ldots, a_r, q$ such that
$(a_1, \ldots, a_r, q)=1$, $1 \le q \ll P^{\Delta}$ and
\[
  |q\alpha_i-a_i|  \ll P^{-d+\Delta} \;\;\; (1 \le i \le r).  
\]
Using the techniques from \cite{B}, one can then establish the announced
asymptotic formula for the number of integer zeros of $F_1=\ldots=F_r=0$
in an expanding box.


\begin{thebibliography}{10}
\bibitem{AM} \textsc{Aleksandrov, A.G. \& Moroz, B.J.}
{Complete intersections in relation to a paper of B.J. Birch,
\textit{Bull. London. Math. Soc.} \textbf{34} (2002), 149--154.}
\bibitem{B} \textsc{Birch, B.J.}
{Forms in many variables,
\textit{Proc. Royal Soc. A} \textbf{265} (1962), 245--263.}
\bibitem{C} \textsc{Colliot-Th\'{e}l\`{e}ne, J.,  Sansuc, J. \&
Swinnerton-Dyer, P.}
{Intersection of two quadrics and Ch\^{a}telet surfaces. I.,
\textit{J. Reine Angew. Math.} \textbf{373} (1987), 37--107.}
\bibitem{DL} \textsc{Davenport, H. \& Lewis, D.J.}
{Non-homogeneous cubic equations, \textit{J. London Math. Soc.}
\textbf{39} (1964), 657--671.}
\bibitem{D} \textsc{Dietmann, R.}
{Systems of rational quadratic forms, \textit{Arch. Math. (Basel)}
\textbf{82} (2004), 507--516.}
\bibitem{HB} \textsc{Heath-Brown, D.R.}
{Zeros of Pairs of Quadratic Forms,
\textit{arXiv:1304.3894}.}
\bibitem{M} \textsc{Meyer, A.}
{\"Uber die Aufl\"osung der Gleichung $ax^2+by^2+cz^2+du^2+eu^2=0$ in ganzen
Zahlen, \textit{Vierteljahresschrift Natur. Ges. Z\"urich} \textbf{29},
220--222 (1884).}
\bibitem{Munshi} \textsc{Munshi, R.}
{Pairs of quadrics in $11$ variables,
\textit{arXiv:1305.1461}.}
\bibitem{Schindler} \textsc{Schindler, D.}
{A variant of Weyl's inequality for systems of forms and applications,
\textit{arXiv:1403.7156}}
\bibitem{S1} \textsc{Schmidt, W.M.}
{Simultaneous rational zeros of quadratic forms,
Seminar of Number Theory, Paris 1980--81, \textit{Progr. Math.}
\textbf{22} (1982), 281--307.}
\bibitem{S2} \textsc{Schmidt, W.M.}
{On cubic polynomials. IV. Systems of rational equations,
\textit{Monatsh. Math.} \textbf{93} (1982), 329--348.}
\bibitem{Schmidt} \textsc{Schmidt, W.M.}
{The density of integer points on homogeneous varieties,
\textit{Acta Math.} \textbf{154} (1985), 243--296.}
\end{thebibliography}
\end{document}